\documentclass{amsart}
\usepackage{ifpdf}
\ifpdf
\else
\fi

\usepackage[utf8]{inputenc}
\usepackage{lmodern}
\usepackage{tikz,tkz-tab}
\usepackage{amsmath,amssymb,amsthm,bbm}
\usepackage[english]{babel}
\usepackage{array,multirow,makecell}
\usepackage{verbatim}

\usepackage{inconsolata}
\usepackage{listings}

\lstdefinelanguage{Julia}%
  {morekeywords={abstract,break,case,catch,const,continue,do,else,elseif,%
      end,export,false,for,function,immutable,import,importall,if,in,%
      macro,module,otherwise,quote,return,switch,true,try,type,typealias,%
      using,while},%
   sensitive=true,%
   alsoother={$},%
   morecomment=[l]\#,%
   morecomment=[n]{\#=}{=\#},%
   morestring=[s]{"}{"},%
   morestring=[m]{'}{'},%
}[keywords,comments,strings]%

\DeclareMathOperator{\Cov}{Cov}

\DeclareMathOperator{\Ber}{Ber}

\newcommand{\regine}[1]{\marginpar{{\footnotesize R: #1}}}

\newcommand{\communique}{\leftrightarrow}

\newcommand{\N}{\mathbb N}

\newcommand{\Z}{\mathbb{Z}}
\newcommand{\Zdeux}{\mathbb{Z}^2}
\newcommand{\Zdeuxstar}{\mathbb{Z}_{*}^2} 
\newcommand{\Zp}{\mathbb{Z}_{*}^2}
\newcommand{\Zd}{\mathbb{Z}^d}
\newcommand{\C}{\mathbb{C}}

\newcommand{\R}{\mathbb{R}}

\renewcommand{\P}{\mathbb{P}}
\newcommand{\E}{\mathbb{E}}

\newcommand\OmegaEP{\Omega_{\textrm{EP}}}

\newcommand{\contour}{\Gamma}

\def\w0{\widetilde{0}}

\def\C{\mathcal{C}}

\newcommand{\espace}{\rule{0pt}{8ex}}

\renewcommand{\epsilon}{\varepsilon}
\renewcommand{\phi}{\varphi}

\newcommand{\ie}{\emph{i.e. }}

\newcommand{\miniop}[3]{%
\renewcommand{\arraystretch}{0.6}
\begin{array}{c}
{\scriptstyle #1}\\
#2\\
{\scriptstyle #3}
\end{array}
\renewcommand{\arraystretch}{1}}

\newcommand{\1}{1\hspace{-1.3mm}1}

\newtheorem{theorem}{Theorem}[section]

\newtheorem{lemma}{Lemma}[section]

\newtheorem{conjecture-dsk}[theorem]{Conjecture}

\newtheorem{lemme}[theorem]{Lemma}

\newtheorem{conj}[theorem]{Conjecture}

\begin{document}

\title[Eulerian percolation]{Does Eulerian percolation on $\mathbb Z^2$ percolate ?}

{
\author{Olivier Garet}
\address{
Universit\'e de Lorraine, Institut \'Elie Cartan de Lorraine, CNRS, UMR 7502, Vandoeuvre-l{\`e}s-Nancy, F-54506, France\\
}
\email{Olivier.Garet@univ-lorraine.fr}
\author{R{\'e}gine Marchand}
\address{
Universit\'e de Lorraine, Institut \'Elie Cartan de Lorraine, CNRS, UMR 7502, Vandoeuvre-l{\`e}s-Nancy, F-54506, France\\
}
\email{Regine.Marchand@univ-lorraine.fr}
\author{Ir\`ene Marcovici}
\address{
Universit\'e de Lorraine, Institut \'Elie Cartan de Lorraine, CNRS, UMR 7502, Vandoeuvre-l{\`e}s-Nancy, F-54506, France\\
}
\email{Irene.Marcovici@univ-lorraine.fr}

}
\def\motsclefs{Eulerian percolation, Ising model, percolation with degree constraints.}


\begin{abstract}
Eulerian percolation on $\Z^2$ with parameter $p$ is the classical Bernoulli bond percolation with parameter $p$ conditioned on the fact that every site has an even degree. We first explain why Eulerian percolation with parameter~$p$ coincides with the contours of the Ising model for a well-chosen parameter~$\beta(p)$. Then we study the percolation properties of Eulerian percolation. Some key ingredients of the proofs  are couplings between Eulerian percolation, the Ising model and FK-percolation.
\end{abstract}

\maketitle

\section{Introduction}
Eulerian percolation with parameter $p$ on the edges of a finite graph is the classical
independent Bernoulli percolation with parameter $p$ on its edges, but conditioned to be even, \ie conditioned to the fact that each vertex of the graph has an even number of open edges touching it. In this paper, we aim to study the percolation properties of the Eulerian (or even) percolation on the edges of $\Z^2$. 
This paper has two parts.

\medskip
1. On $\Z^2$, the event by which we want to condition has probability $0$. The first step is thus to define properly the Eulerian percolation measures on the edges of~$\Z^2$, by the mean of specifications in finite boxes and of Gibbs measures. Doing so, the Eulerian percolation measure with parameter $p$ is given by the contours of the Ising model on the sites of the dual $\Z^2_*\sim \Z^2$ for a well-chosen parameter $\beta=\beta(p)$:
\begin{theorem}
\label{THEO-mesure}
For every $p \in [0,1]$, there exists a unique even percolation measure on the edges of $\Z^2$ with opening parameter $p$, and we denote it by $\mu_p$. It is the image by the contour application of any Gibbs measure for the Ising model on the dual graph  $\Z^2_*$ of $\Z^2$, with parameter 
$$\beta=\beta(p)=\frac12\log\frac{1-p}{p} \quad \Leftrightarrow p=\frac{1}{1+\exp(2\beta)}.$$ 
Moreover, $\mu_p$ is invariant and ergodic under the natural action of $\Z^2$.
\end{theorem}

Note that Eulerian percolation with parameter $p<1/2$, resp. $p>1/2$, corresponds to the contours of the Ising model in the ferromagnetic range $\beta>0$, resp. antiferromagnetic range $\beta<0$.
Theorem \ref{THEO-mesure} is an extension of Theorem~5.2 of Grimmett and Janson~\cite{GrimJans}, that studies random even subgraphs on finite planar graphs. In the same paper, they mention the existence of a thermodynamic limit, but the question of uniqueness is not asked.

\medskip
2. We are interested in the probability, under the even percolation measure $\mu_p$, of the percolation event
$$\C=\text{``there exists an infinite open cluster"}.$$
Our first result consists in proving the almost-sure uniqueness of the infinite cluster when it exists: 
\begin{theorem} \label{THEO:uni}
For every $p \in [0,1]$, there exists $\mu_p$-almost surely exactly one infinite cluster or $\mu_p$-almost surely no infinite cluster.
\end{theorem}
Note that the ``even degree" condition induces correlations between states of edges, that break the classical finite energy property. However, we can adapt the classical proof by using the interpretation in terms of contours of the Ising model. To study the percolation itself, we have at our disposal the results proved for the Ising model on $\Z^2$, especially in the ferromagnetic range. Remember that $\beta_c={1\over 2} \log (1+\sqrt{2})$ is the critical value of the Ising model in $\Z^2$; we introduce the corresponding percolation parameter
$$p_{c,\mathrm{even}}=\frac{1}{1+\exp(2\beta_c)}  =1-\frac{1}{\sqrt 2} <\frac12.$$
We prove the following: 
\begin{theorem} \label{THEO:perco}
In terms of even percolation with parameter $p\in [0,1]$, 
\begin{itemize}
\item for every $p \in[0,p_{c,\mathrm{even}}]$, $\mu_p(\C)=0$,
\item for every $p \in (p_{c,\mathrm{even}}, 1] \backslash\{1-p_{c,\mathrm{even}}\}$, $\mu_p(\C)=1$.
\end{itemize}
In terms of the Ising model with parameter $\beta\in\R$, these results correspond to:
\begin{itemize}
\item for $\beta\geq\beta_c$, for every Gibbs measure with parameter $\beta$, contours a.s. do not percolate,
\item for $\beta<\beta_c$ such that $\beta\not=-\beta_c$, for every Gibbs measure with parameter $\beta$, contours a.s. percolate.
\end{itemize}
\end{theorem}
These results are summarized in the following table:
$$
\begin{array}{|l|lccccccccr|}
\hline
\multirow{2}*{$p$} & \multirow{2}*{$0$} && p_{c,\mathrm{even}} && \multirow{2}*{1/2} &&1-p_{c,\mathrm{even}} &  & \multirow{2}*{1} &\\
& && =1-1/ \sqrt{2} && && =1/ \sqrt{2}& & & \\
\hline
\beta(p) & +\infty && \beta_c && 0 && -\beta_c & & -\infty &\\
\hline
\mu_p && \text{no perco.} & ] &  & \text{perco.}  & & $?$ &\text{perco.} & &\\
\hline
\end{array}
$$

\medskip
We did not manage to settle the case $p=1-p_{c,\mathrm{even}}$ (corresponding to $\beta=-\beta_c$ for the Ising model). In independent Bernoulli bond percolation, $p \mapsto \P_p(\C)$ is non-decreasing, and this follows from a natural coupling of percolation for all parameters $p\in[0,1]$. The same monotonicity occurs for FK percolation with parameter $q\ge 1$. This is strongly related to the fact that FK percolation satisfies the FKG inequality. Here, conditioning by the Eulerian condition breaks the association, even if the underlying graph is Eulerian. See Section~\ref{leschosesetranges} for an example of the strange things that may happen.
We naturally conjecture that $p_{c,\mathrm{even}}$ is indeed the unique percolation threshold for Eulerian percolation on $\Z^2$: 
\begin{conj} $\;$

$\bullet$ In terms of even percolation: $\mu_{1-p_{c,\mathrm{even}}}(\C)=1$. 

$\bullet$ In terms  of the Ising model: for every Gibbs measure with parameter $-\beta_c$, contours a.s. percolate.
\end{conj}

The percolation results for $p \le 1/2$ essentially follow from the results about percolation of colors in the Ising model in the ferromagnetic case $\beta>0$. The Ising model in the antiferromagnetic case has been much less studied, so other kinds of arguments are needed for $p>1/2$. 

In order to settle the case $1/2<p<1-p_{c,\mathrm{even}}$, we introduce a coupling between $\mu_p$ and $\mu_{1-p}$. This coupling has the property to increase the connectivity, so that  percolation for $p\in(p_c,1/2)$ implies  also percolation for $p\in(1/2, 1-p_{c,\mathrm{even}})$.

The case for $p>1-p_{c,\mathrm{even}}$ follows from the link between the Ising model and FK percolation. A stochastic comparison between even percolation
with a large parameter and independent percolation with a large parameter gives the result for $p\ge 3/4$, and we only sketch the extension of that proof to $p\in(1-p_{c,\mathrm{even}},1]$, using techniques from Beffara and Duminil-Copin~\cite{MR2948685}.

\section{Eulerian percolation probability measures}

On $\Z^2$, we consider the set of edges $\E_2$ between vertices at distance $1$ for $\|.\|_1$. An edge configuration is an element
$\omega \in \{0,1\}^{\E_2}:$
if $\omega(e)=1$, the edge $e$ is present (or open) in the configuration $\omega$, and if $\omega(e)=0$, the edge is absent (or closed).
For $x \in \Z^2$, we define the degree $d_\omega(x)$ of $x$ in the configuration $\omega$ by setting
$$d_x(\omega)=\sum_{e\ni x} \omega(e).$$
An Eulerian edge configuration is then an element of
$$\OmegaEP=\{\omega\in\{0,1\}^{\E_2}: \; \forall x\in\Zd\,,\; d_x(\omega)=0\; [2]\}.$$
If $\omega, \eta \in \OmegaEP$ and $\Lambda \subset \E_2$, we denote by $\eta_{\Lambda}\omega_{\Lambda^c}$ the concatenation of the configuration $\eta$ restricted to $\Lambda$ and of the configuration $\omega$ restricted to $\Lambda^c$.

\subsection*{Gibbs measures for Eulerian percolation}

For each finite subset $\Lambda$ of $\E_2$ and each function $f$ on $\OmegaEP$, we can define 
\begin{align}
\forall \omega\in\OmegaEP \quad (M^p_{\Lambda}f)(\omega) & =\sum_{\eta_{\Lambda}\in\{0,1\}^{\Lambda}}\1_{\OmegaEP}(\eta_{\Lambda}\omega_{\Lambda^c})f(\eta_{\Lambda}\omega_{\Lambda^c}) \left( \frac{p}{1-p} \right)^{\sum_{e\in\Lambda}\eta_e}, \nonumber \\
(\mu^p_{\Lambda}f)(\omega) & = \frac{(M^p_{\Lambda}f)(\omega)}{(M^p_{\Lambda}1)(\omega)}. \label{compo}
\end{align}
Note that $\mu^p_{\Lambda}$ is Feller, in the following sense:
$\mu^p_{\Lambda}f$ is continuous (for the product topology) as soon as $f$ is continuous.
A standard calculation gives
\begin{align*}
  \mu^p_{\Delta}\circ \mu^p_{\Lambda}=\mu^p_{\Delta}\text{ for }\Lambda\subset\Delta.
\end{align*}
We denote by $\mu_{\Lambda,\omega}^{p}$ the probability measure on
$\{0,1\}^{\E_2}$ that is such that, for each bounded measurable function~$f$, 
$$\int_{\OmegaEP} f\ d\mu_{\Lambda,\omega}^p=(\mu^p_{\Lambda}f)(\omega)=\frac{\miniop{}{\sum}{\eta_{\Lambda}\in\{0,1\}^{\Lambda}}\1_{\OmegaEP}(\eta_{\Lambda}\omega_{\Lambda^c})f(\eta_{\Lambda}\omega_{\Lambda^c}) \left( \frac{p}{1-p} \right)^{\sum_{e\in\Lambda}\eta_e}}{\miniop{}{\sum}{\eta_{\Lambda}\in\{0,1\}^{\Lambda}}\1_{\OmegaEP}(\eta_{\Lambda}\omega_{\Lambda^c}) \left( \frac{p}{1-p} \right)^{\sum_{e\in\Lambda}\eta_e}}.$$
A probability measure $\mu$ on $(\{0,1\}^{\E_2},\mathcal{B}(\{0,1\}^{\E_2}))$ is said to be a Gibbs measure for Eulerian percolation (or a Eulerian percolation probability measure) if one has
\begin{itemize}
\item $\mu(\OmegaEP)=1$
\item For each  continuous fonction on $\{0,1\}^{\E_2}$, for each finite subset $\Lambda$ of $\E_2$, 
\begin{equation}
\label{eqgibbsmu}
  \int_{\OmegaEP} f\ d\mu =\int_{\OmegaEP} (\mu^p_{\Lambda} f)d\mu.
\end{equation}  
\end{itemize}
We denote by $\mathcal{G}_{\textrm{EP}}(p)$ the set of Gibbs measures for Eulerian percolation with opening parameter $p$.

\subsection*{Colorings with two colors and Eulerian percolation}

A natural way to obtain a Eulerian configuration of the edges of a planar graph is to take the contours of a coloring in two colours of the sites of its dual, and this is what we decribe now in the $\Z^2$ case.

Let $\Zp=(1/2,1/2)+\Z^2$ be the dual graph of $\Z^2$. The set $\E^2_*$ of edges of $\Zp$ is the image of $\E^2$ by the translation with respect to the vector $(1/2,1/2)$. If $e \in \E^2$, we denote by $e_*$ its dual edge, \ie the only edge in $\E^2_*$ that intersects $e$. We can map any coloring of the sites of $\Zp$ with the two colors $-1$ and~$1$ to its contour in the following way:
$$
\begin{array}{lrcl} 
\contour: &\{-1,1\}^{\Z^2_*} & \longrightarrow & \Omega_{\textrm{EP}} \\
&\sigma=(\sigma_{i_*})_{{i_*}\in\Zp} & \longmapsto & (\eta_e)_{e\in\E^2}, \text{ with } \eta_{e}=\1_{\{\sigma_{i_*}\ne \sigma_{j_*}\}} \text{ if } e_*=\{i_*,j_*\}
\end{array}.$$
Let us see that $\contour(\sigma) \in \Omega_{\textrm{EP}}$. Indeed, set $\eta=\contour(\sigma)$, and fix $x \in \Z^2$. Let 
$a_*,b_*,c_*,d_*$ be the four corners of the square with length side $1$ in  $\Z^2_*$ whose center is $x$: then the four edges issued from $x$ are the dual edges of $\{a_*,b_*\}$, $\{b_*,c_*\}$, $\{c_*,d_*\}$ and $\{d_*,a_*\}$. Thus
\begin{align*}
(-1)^{d_x(\eta)} 
& = (-\sigma_{a_*}\sigma_{b_*})(-\sigma_{b_*}\sigma_{c_*})(-\sigma_{c_*}\sigma_{d_*})(-\sigma_{d_*}\sigma_{a_*})=1.
\end{align*}
So $\contour(\sigma) \in \Omega_{\textrm{EP}}$.

Reciprocally, the dual of a planar Eulerian graph is bipartite (see for instance Wilson and Van Lint~\cite{MR1871828}, Theorem (34.4)
), and there are exactly two ways of coloring the sites of a connected bipartite graph with two colors in such a way that the extremities of every edge are in different colors. In our $\Z^2$ case, fix a Eulerian edge configuration $\eta$. By setting  $c_{\eta}(0_*)=+1$, and  for any $x_*\in \Z^2_*$,  $c_{\eta}(x_*)$ equals (-1) power  the number of edges in $\eta$ crossed by any path (in the dual) between $0_*$ and $x_*$, we properly define a coloring $c_\eta$ of $\Z^2_*$, and 
$\contour^{-1}(\eta)=\{c_\eta, -c_\eta\}.$ Finally, the contour application $\contour$ is surjective and two-to-one.

As we will see now, the Gibbs measures for Eulerian percolation can be obtained
as the images by the contour application $\contour$ of the Gibbs measures for the Ising model in $\Z_2^*$. 

\subsection*{Gibbs measures for the Ising model on $\Z^2_*$}
It is of course the same model as the Ising model on $\Z^2$, but to avoid  confusion between the initial graph $\Z^2$ and its dual $\Z^2_*$ in the sequel, we present it directly in the dual $\Z^2_*$.
Fix a parameter $\beta \in \R$.
\newcommand{\Sq}{\{-1,+1\}}
For a finite subset $\Lambda$ of $\Zdeuxstar$, the Hamiltonian on  $\Lambda$ is defined by
$$\forall \omega \in \Sq^{\Zdeuxstar} \quad H_{\Lambda}(\omega)=-\sum_{\substack{e=\{x,y\}\in\E^2_*\\ e\cap\Lambda\ne\varnothing}} \omega_x\omega_y.$$
Then, we can define, for each bounded measurable function $f$,
\begin{align*}
\forall \omega \in \Sq^{\Zdeuxstar} \quad 
\mathcal{Z}^{\beta}_{\Lambda}(\omega) & =\sum_{\eta\in \Sq^{\Lambda}}\exp(-\beta H_{\Lambda}(\eta_{\Lambda}{\omega}_{\Lambda^c})), \\
{\Pi}^{\beta}_{\Lambda}f({\omega}) & =\frac1{\mathcal{Z}^{\beta}_{\Lambda}({\omega})}
{\miniop{}{\sum}{\eta\in \Sq^{\Lambda}}\exp(-\beta H_{\Lambda}(\eta_{\Lambda}{\omega}_{\Lambda^c}))
f(\eta_{\Lambda}{\omega}_{\Lambda^c})}.
\end{align*}
For each $\omega$, we denote by ${\Pi}^\beta_{\Lambda,\omega}$ the 
probability measure on $\Sq^{\Zdeuxstar}$ which is associated to the map $f\mapsto \Pi^\beta_{\Lambda} f(\omega)$.
When $\beta=0$, colors of sites inside $\Lambda$ are i.i.d. and follow the uniform law in $\{-1,+1\}$. When $\beta>0$, neighbour sites prefer to be in the same color (ferromagnetic case), while when $\beta<0$, neighbour sites prefer to be in different colors (anti-ferromagnetic case).

A Gibbs measure for the Ising model on $\Z^2_*$ with parameter $\beta$ is any probability measure $\gamma$ on $\{-1,+1\}^{\Z^2_*}$ such that for each continuous function, for each finite subset $\Lambda$ of $\Zdeuxstar$, 
$$\int_{\{-1,+1\}^{\Z^2_*}} f\ d\gamma =\int_{\{-1,1\}^{\Z^2_*}} (\Pi^\beta_{\Lambda} f) d\gamma.$$
We denote by $\mathcal{G}(\beta)$ the set of Gibbs measures for the Ising model with parameter $\beta$.
The Ising model presents a phase transition: set $\beta_c=\frac12\log(1+\sqrt 2)$ (see Onsager~\cite{MR0010315}), then 
\begin{itemize}
\item if $0 \le \beta \le  \beta_c$, then there is a unique Gibbs measure;
\item if $\beta>\beta_c$ then there are infinitely many Gibbs measures. The set $\mathcal{G}(\beta)$ is the convex hull of two extremal measures $\gamma_\beta^+$ and $\gamma_\beta^-$, that can be deduced one from the other by exchanging the two colors. This result has been obtained independently by Aizenmann~\cite{MR573615} and Higuchi~\cite{MR712693}. See also Georgii--Higuchi~\cite{MR1757954}.
\end{itemize} 
For $\beta<0$, the Gibbs measures are obtained from $\mathcal{G}(-\beta)$ by changing the colors on the subset of even sites. In other words, if
$$S((\omega)_{\omega\in\Zdeuxstar})=((-1)^{i+j} \omega_{(i,j)})_{(i,j)\in\Zdeuxstar},$$
then $\mu_S=(A\mapsto\mu(S^{-1}(A))$ belongs to $\mathcal{G}(-\beta)$ if and only if $\mu\in\mathcal{G}(\beta)$. For the details, see Chapter~6 in Georgii~\cite{MR956646}.
\begin{figure}  
  \centering
  \begin{tabular}{cc}
   \includegraphics[scale=0.5]{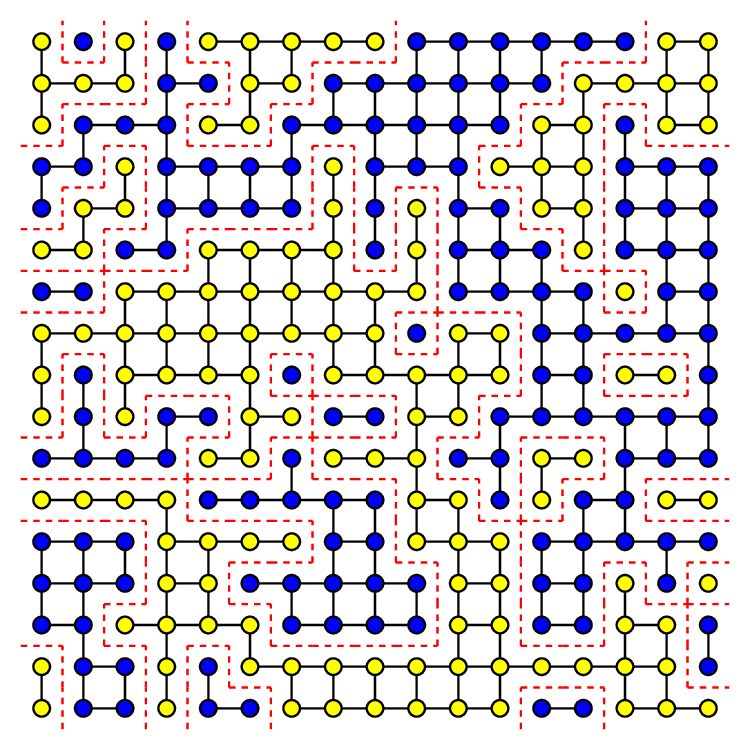}&\includegraphics[scale=0.5]{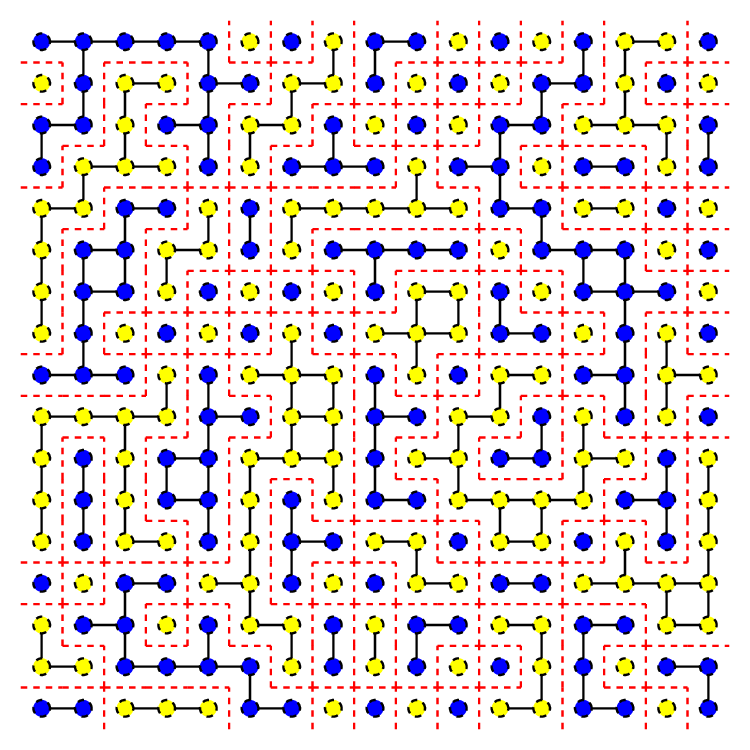}
   \end{tabular} 
  \caption{The mapping $\beta\longleftrightarrow -\beta$}
\end{figure}

\subsection*{Proof of Theorem \ref{THEO-mesure}, existence}
We first prove the existence of Gibbs measure for Eulerian percolation.
Let us define $\Lambda_n=(1/2,1/2)+\{-n,\dots,n\}^2 \subset \Zdeuxstar$ 
and denote by $E(\Lambda_n)$ the set of edges $e$ such that $e_*$ has at least one end in $\Lambda_n$.
Since $\OmegaEP$ is a closed subset of the compact set $\{0,1\}^{\E^2}$, the sequence  $(\mu^p_{E(\Lambda_n),0})_n$ has a limit point $\mu$ with $\mu(\OmegaEP)=1$.
Using Equation~\eqref{compo} and the fact that  $\mu^p_{\Lambda}$ is Feller, it is easy to see that $\mu\in\mathcal{G}_{\textrm{EP}}(p)$, which is therefore not empty. 

This proof is not surprising for people who are familiar to the general theory of Gibbs measures, as described in Georgii~\cite{MR956646}.
Nevertheless, it must be noticed that  $\mu^p_{\Lambda}f$ is not defined on the whole set $\{0,1\}^{\E^2}$ (it is not a specification in the realm of Georgii~\cite{MR956646}), which leads us to mimic a standard proof.
\hfill $\square$

\medskip
To prove the uniqueness of the even percolation probability measure, we first need  the following lemma:

\begin{lemma}
\label{lemmeimage}
Let $c \in \{-1,1\}^{\Zdeuxstar}$ and $\eta\in\OmegaEP$ with $\eta=\contour(c)$. Suppose that $\Lambda_*$ is a simply connected subset of 
$\Zdeuxstar$, and denote by $E(\Lambda_*)$ the set of edges $e$ such that $e_*$ has at least one end in $\Lambda_*$. 

Fix $p \in(0,1)$ and set $\beta=\beta(p)=\frac12\log\frac{1-p}{p}$.
Then, the probability $\mu^p_{E(\Lambda_*),\eta}$ is the image of  $\Pi^{\beta(p)}_{\Lambda_*,c}$ under the contour application $\omega\mapsto\contour(\omega)$.
\end{lemma}
  
\begin{proof}
By construction, the image of  $\Pi^{\beta}_{\Lambda_*,c}$ under the map $\omega\mapsto\contour(\omega)$ is concentrated on configurations that coincide with $\eta$ outside $E(\Lambda_*)$. Obviously it is the same for $\mu^p_{E(\Lambda_*),\eta}$, so we must focus on the behaviour of the edges in $E(\Lambda_*)$.
      
Let $\eta'\in\OmegaEP$ be such that $\eta$ and $\eta'$ coincide outside $E(\Lambda_*)$.
There are exactly two colorings $c', -c'$ such that $\contour(c')=\contour(-c')=\eta'$. 
If $x$ and $y$ are two neighbours in $(\Lambda^*)^c$, then
$$c_xc_y=1-2\eta_{(x,y)_*}=1-2\eta'_{(x,y)_*}=c'_xc'_y,$$
so $c_x c'_x=c_y c'_y$.
Since $\Lambda_*^c$ is connected, it follows that one of the two colorings, say $c'$, coincides with $c$ on $(\Lambda_*)^c$ (and $-c'$ with $-c$). Thus $\Pi^{\beta}_{\Lambda_*,c}(-c')=0$ and $\Pi^{\beta}_{\Lambda_*,c}(c')>0$, and:
\begin{align*}
& \Pi^{\beta}_{\Lambda_*,c}(\contour(.)=\eta')= \Pi^{\beta}_{\Lambda_*,c}(c') \\
& = \frac1{\mathcal{Z}^{\beta}_{\Lambda_*}(c)} \exp \left( \beta \sum_{\substack{e=\{x,y\}\in\E^2_*\\ e\cap \Lambda_*\ne\varnothing}}  c'_xc'_y\right) 
 = \frac1{\mathcal{Z}^{\beta}_{\Lambda_*}(c)} \exp \left( \beta \sum_{\substack{e=\{x,y\}\in\E^2_*\\ e\cap \Lambda_*\ne\varnothing}} (1-2\eta'_{(x,y)_*}) \right) \\
& = \frac1{\mathcal{Z}^{\beta}_{\Lambda_*}(c)} \exp \left( \beta\sum_{e\in E(\Lambda_*)} (1-2\eta'_e) \right)
 = \frac{\exp(\beta|E(\Lambda_*)|)}{\mathcal{Z}^{\beta}_{\Lambda_*}(c)} \left(\frac{p}{1-p}\right)^{\sum_{e\in  E(\Lambda_*)}\eta'_e}\\
&=\alpha_{\Lambda_*,\eta}\mu^p_{E(\Lambda_*),\eta}(\eta').
\end{align*}
Since we compare probability measures with the same support, $\alpha_{\Lambda^*,\eta}=1$.
\end{proof}

\subsection*{Proof of Theorem \ref{THEO-mesure}, uniqueness}
Let us now see that all Gibbs measures for the Ising model with parameter $\beta$ have the same image by the application $\contour$. Let $\gamma\in\mathcal{G}(\beta)$: there exists $\alpha\in [0,1]$ such that $\gamma=\alpha\gamma_\beta^+ +(1-\alpha)\gamma_\beta^-$. Remember that $\gamma_\beta^-$ is the image of $\gamma_\beta^+$ by the exchange of colors, that leaves the contours unchanged.
So,
if $A\in\mathcal{B}(\{0,1\}^{\E^2})$, 
\begin{align*}
\gamma(\contour\in A) 
& = \alpha\gamma_\beta^+(\contour \in A) + (1-\alpha)\gamma_\beta^-(\contour \in A) \\
& = \alpha\gamma_\beta^+(\contour \in A) + (1-\alpha)\gamma_\beta^+(\contour \in A) = \gamma_\beta^+(\contour\in A)
\end{align*}

Let $\mu\in\mathcal{G}_{EP}(p)$ and set as before $\Lambda_n=(1/2,1/2)+\{-n,\dots,n\}^2$. Let $f$ be a  continuous function on $\{-1,1\}^{\E^2}$, and let us prove that for each $\eta\in\OmegaEP$, 
$$(\mu^p_{E(\Lambda_n)}f)(\eta)\to \int f\circ\contour \ d\gamma_\beta^+.$$
With Equation~\eqref{eqgibbsmu}, it will imply by dominated convergence that
$$\int_{\OmegaEP}f d \mu=\int f\circ\contour \ d\gamma_\beta^+,$$
and thus that $\mu$ is the image by the application $\contour$ of $\gamma_\beta^+$, or of any Gibbs measure for the Ising model with parameter $\beta$.

Let $\eta\in\OmegaEP$ be an Eulerian edge configuration, and let $c \in \{-1,+1\}^{\Zdeuxstar}$ be such that $\contour(c)=\eta$.
Let $x$ be a limiting value of  $((\mu^p_{E(\Lambda_n)}f)(\eta))_{n\ge 1}$.
By extracting a subsequence if necessary, we can assume that $(\Pi_{\Lambda_{n},c})_{n\ge 1}$ converges to $\gamma$, which is then in $\mathcal{G}(\beta)$, and that $x=\displaystyle \lim_{n\to +\infty}(\mu^p_{E(\Lambda_{n})}f)(\eta)$.
By Lemma~\ref{lemmeimage},
\begin{align*}
(\mu^p_{E(\Lambda_{n})}f)(\eta) & =\Pi^{\beta}_{\Lambda_{n},c_\eta}(f\circ\contour), \\
\text{so } x & =\int_{\{-1,1\}^{\Zdeuxstar}} (f\circ\contour)\ d\gamma = \int_{\{-1,1\}^{\Zdeuxstar}} (f\circ\contour) d\gamma_\beta^+.
\end{align*}
To conclude, note that $\gamma_\beta^+$ is stationary and ergodic, and so does $\mu_p$.\hfill $\square$

\section{Unicity of the infinite cluster in Eulerian percolation}

\subsection*{Proof of Theorem \ref{THEO:uni}.}
Since $\mu_p$ is ergodic and $\C$ is a translation-invariant event, it is
obvious that $\mu_p(\C)\in\{0,1\}$. To prove the unicity of the infinite cluster, we now follow the famous proof by Burton and Keane \cite{MR990777}.
The main point here is that the Eulerian percolation measure does not satisfy the finite energy property: once a configuration is fixed outside a box, the even degree condition forbids some configurations inside the box. But the Ising model has the finite energy property, and we will thus use the representation of even percolation in terms of contours of the Ising model. 

The number $N$ of infinite clusters is translation-invariant, so the ergodicity of $\mu_p$ implies that it is $\mu_p$-almost surely constant: there exists $k \in \N \cup\{\infty\}$ such that $\mu_p(N=k)=1$. The first step consists in proving that $k\in \{0,1,\infty\}$.
So assume for contradiction that  $k$ is an integer larger than $2$. Consider a finite box $\Lambda$, large enough to ensure that with positive probability (under $\mu_p$), the box $\Lambda$ intersects at least two infinite clusters. Using Theorem \ref{THEO-mesure}, this implies that with positive probability (under $\gamma_\beta^+$ for the parameter $\beta$ corresponding to $p$), the contours of the Ising model present two infinite connected components that intersect $\Lambda$. But the Ising model has the finite energy property: by forcing the colors inside $\Lambda$ to be a chessboard, we keep an event with positive probability, and we decrease the number of infinite clusters in the contours by at least one. Coming back to Eulerian percolation, this gives $\mu_p(N \le k-1)>0$, which is a contradiction. See \cite{MR620606,MR648202} for the first version of such an argument.

\medskip

In the final step, we prove that $k=\infty$ is impossible. Assume by contradiction that $\mu_p(N=+\infty)=1$. We work now with the colorings of the sites of $\mathbb Z^2_*$, under~$\gamma_{\beta(p)}$.

By taking $L \in \N$ large enough, we can assume that the event $E_L$ ``the box $B_L=[-L,L]^2$ intersects at least $30$ infinite clusters" has positive probability.  Let $\partial \eta_0$ be a coloring of the sites in $\partial_{int}B_L=B_L \backslash B_{L-1}$ such that
$$\gamma_{\beta(p)}(\eta \in E_L, \, \eta_{|\partial_{int}B_L}=\partial \eta_0)>0.$$
Take $\omega$ in this event. Each infinite cluster intersecting $B_L$ crosses
$\partial_{int}B_L$ via an open edge, and this edge sits between a $+1$ site and a $-1$ site. 

Thus the $30$ distinct infinite (edge) clusters intersecting $B_L$ imply the existence of
at least $15$ clusters of $+1$ vertices in $\partial_{int}B_L$. To avoid geometric intricate details, we do not want to consider  $+1$-clusters in $\partial_{int}B_L$  that are in the corners: we thus remove from our  $15$ clusters at most $12=3\times 4$ clusters (the one containing the corner if it is a $+1$, and the nearest $+1$ cluster on each side). We are now left with at least $3$ disjoint $+1$-clusters in $\partial_{int}B_L$, sitting near edges of distinct infinite clusters: they are far away enough so that we can draw, inside $B_{L}$, $3$ paths of sites linking these three clusters to three of the four centers of the sides of $\partial_{int}B_2$, in such a way that two distinct paths are not $*$-connected. See Figure \ref{trifurcation}.

\begin{figure}
\begin{center}
\includegraphics{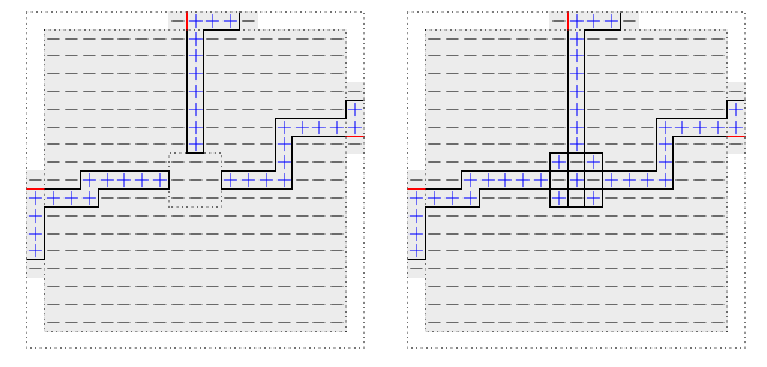}
\end{center}
\caption{Construction of a trifurcation in $B_{L_1}$. Dotted squares are, from inside to outside, $B_{1}$, $B_{L-1}$ and $B_L$. Red edges are, on the left, in three distinct infinite clusters of open edges.} \label{trifurcation}
\end{figure}  

Consider now the following coloring of $B_{L-1}$: all sites in the three paths are $+1$, all the other sites are $-1$. With this coloring, $B_{L-1}$ intersects exactly three infinite clusters of open edges. If we change the coloring of $B_{1}$ in a chessboard, $B_{L-1}$ intersects exactly one infinite cluster of open edges. In this case, we say that $0$ is a trifurcation.
As $\gamma_{\beta(p)}$ has finite energy, we see that the probability that $0$ is a trifurcation has positive probability, and the end of the proof is as in Burton-Keane.

\section{Percolation properties of Eulerian percolation}

The proof of Theorem \ref{THEO:perco} is split into five steps: Lemmas \ref{LEM:un}, \ref{LEM:critique}, \ref{LEM:deux}, \ref{LEM:undeux} and~\ref{LEM:quatre}, that are respectively considering the ranges $(0,p_{c,\mathrm{even}})$, $(p_{c,\mathrm{even}},1/2]$, $(1/2, 1-p_{c,\mathrm{even}})$ and $(1-p_{c,\mathrm{even}},1)$.

\subsection{The ferromagnetic zone of the Ising model: $p\le 1/2$.}


\begin{lemme} \label{LEM:un}
 For $p\in(0,p_{c,\mathrm{even}})$, $\mu_p(\C)=0$.
\end{lemme}

\begin{proof} Let $\omega$ be a spin configuration of $\{+1,-1\}^{\Z^2_*}$, and let $\eta=\Gamma(\omega)$ be the even subgraph of $\Z^2$ made of the contours of $\omega$.

We need here the notion of $*$-neighbours: two sites $x_*,y_* \in \Z^2_*$ are $*$-neighbours if and only if $\|x_*-y_*\|_\infty =1$. A $*$-chain is then a sequence of sites in $\Z^2_*$ such that two consecutive sites are $*$-neighbours.

Let us assume that $\eta$ contains an infinite path $\gamma$. For each edge along $\gamma$, there is a spin $+1$ in the configuration $\omega$ on one side of that edge, and a spin $-1$ on the other side. The set of spins $+1$ (resp. $-1$) in $\omega$ along $\gamma$ constitutes an infinite $*$-chain of spins $+1$ (resp. $-1$), as illustrated in Figure~\ref{figure:evol_chemin}, which shows the evolution of the $*$-chain of spins $+1$ for the different possible steps taken by $\gamma$. Set 
$$\C_*^+=\{\omega\in\{+1,-1\}^{\Z^2_*}: \mbox{ there is an infinite $*$-chain of spins $+1$ in } \omega\}.$$ 
\begin{figure}
\begin{center}
\includegraphics{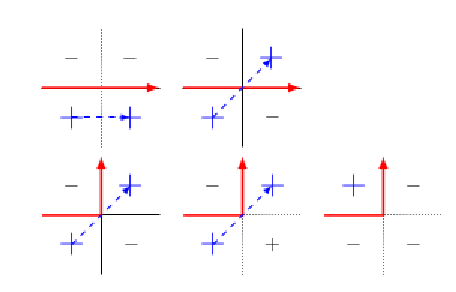}
\end{center}

  \caption{Construction of an $*$-chain of spins $+1$ (dotted arrow) from a infinite path $\gamma$ (full arrow)}\label{figure:evol_chemin}
\end{figure}

It follows from Theorem \ref{THEO-mesure} that for any $p\in (0,1)$,  $\mu_p(\C)\leq \gamma_\beta^-(\C_*^+)$, where $p$ and~$\beta$ are related through the relation $\beta={1\over 2} \log {1-p\over p}$.
By Proposition 1 in Russo~\cite{russo}, we know that if $\beta>\beta_c$, $\gamma_\beta^-(\C_*^+)=0$. It follows that for $p<p_{c,\textrm{even}}$, $\mu_p(\C)=0$.
\end{proof}

\begin{lemme}\label{LEM:critique} At the critical point $p_{c,\mathrm{even}}$, we have:
$\mu_{p_{c,\mathrm{even}}}(\C)=0$.
\end{lemme}

The proof of this lemma will follow from the stochastic comparison between even percolation and the random cluster model stated in Lemma~\ref{LEM:tree}, and can be found just after Lemma \ref{LEM:quatre}. The argument, extended to $p \in [0,  p_{c,\mathrm{even}}]$, also gives an alternative proof of Lemma \ref{LEM:un}.


\begin{lemme} \label{LEM:deux}
For $p\in(p_{c,\mathrm{even}},1/2]$, $\mu_p(\C)=1$.
\end{lemme}

\begin{proof} Let us set $$\C^+=\{\omega\in\{+1,-1\}^{\Z^2_*}:\;  \mbox{ there is an infinite } \mbox{chain of spins $+1$ in } \omega\}.$$

Let $\omega\in\C_*^+\cap (\C^+)^c$, and let $\delta$ be an infinite $*$-chain of spins $+1$ in $\omega$. For each spin $+1$ along $\delta$, let us consider the cluster of spins $+1$ to which it belongs. Since $\omega\not\in \C^+$, these clusters are finite. The union of the contours of these clusters is an infinite connected subgraph of $\Z^2$. Indeed, let $x_1, x_2\in \Z^2_*$ be the coordinates of two consecutive spins $+1$ of the $*$-chain $\delta$. If $\omega(x_1)$ and $\omega(x_2)$ are not in the same cluster of spins $+1$, it means that the step from $x_1$ to $x_2$ in $\delta$ is diagonal (with spins $-1$ in the opposite diagonal), and that the contours of the clusters of $\omega(x_1)$ and $\omega(x_2)$ meet at point $(x_1+x_2)/2$. Thus, any two consecutive points of $\delta$ are such that the contours of their clusters are connected (or possibly the same). By induction, one can then prove that the union of the contours of all the clusters of spins $+1$ of $\delta$ is a connected subgraph of $\Z^2$. 

It follows from Theorem \ref{THEO-mesure}
that for any $p \in (0,1)$, $\mu_p(\C)\geq \gamma_\beta^+(\C_*^+\cap (\C^+)^c)$.
For $\beta\in[0,\beta_c)$, we have $\gamma_\beta^+=\gamma_\beta^-=\gamma_\beta$, and 
\begin{itemize}
\item $\gamma_\beta^+(\C^+)=0$, by Proposition 1 in~\cite{coniglio},
\item $\gamma_\beta^+(\C_*^+)=1$, by Theorem 1 in~\cite{higuchi}.
\end{itemize}
Thus, $\gamma_\beta^+(\C_*^+\cap (\C^+)^c)=1$. It follows that for $p\in(p_c,1/2]$, $\mu_p(\C)=1$.
\end{proof}

\subsection{The antiferromagnetic zone of the Ising model: $p> 1/2$}
This is the most complex case, because the geometry of the antiferromagnetic Ising model is less well known.

\subsubsection{Percolation for $p\in (1/2,1/\sqrt{2})$}

We want here to build, for $p<1/2$, a coupling between $\mu_p$ and $\mu_{1-p}$ that increases connectivity:
\begin{lemme} \label{LEM:domstoch}
  Let $p\in (0,1/2)$. The law of the field  $(\1_{\{x\communique y\}})_{(x,y)\in \Zdeux\times\Zdeux}$ under $\mu_p$ is stochastically dominated by the law of the field  $(\1_{\{x\communique y\}})_{(x,y)\in \Zdeux\times\Zdeux}$ under $\mu_{1-p}$.
\end{lemme}
By Lemma \ref{LEM:deux}, there is percolation under $\mu_p$ for $p\in (p_{c,\mathrm{even}},1/2]$, so we deduce:
\begin{lemme} \label{LEM:undeux}
For $p\in(1/2,1-p_{c,\mathrm{even}})$, $\mu_p(\C)=1$.
\end{lemme}

\begin{proof}[Proof of Lemma \ref{LEM:domstoch}]
For every site $x \in \Z^2+(1/2,1/2)$, we consider the set $E_x \subset \E_2$ of its four surrounding edges, \ie the four edges of the unit square with center~$x$. Define $G_2=\{(x_1,x_2)\in\Zdeux+(1/2,1/2): \;  x_1+x_2\in 2\Z\}$. Then $\E_2$  is the disjoint union of the $E_x$ for $x \in G_2$.

We define $\Omega_x=(\{0,1\}\times \{0,1\})^{E_x}$. A point $(\omega_e,\tilde \omega_e)_{e \in E_x} \in (\{0,1\}\times \{0,1\})^{E_x}$ encodes two configurations of the four edges surrounding $x$: $(\omega_e)_{e \in E_x}$ and $(\tilde \omega_e)_{e \in E_x}$. 

For $(\omega_e)_{e \in E_x}$ let us set $|\omega|=\sum_{e\in E_x} {\omega_e}$. 

\medskip
1. We first define a probability measure $P$ on $\Omega_x=(\{0,1\}\times \{0,1\})^{E_x}$, whose first marginal is $\Ber(p)^{\otimes E_x}$, and whose second marginal is $\Ber(1-p)^{\otimes E_x}$. This probability $P$ is defined by the table below, and has the property that $P$-almost surely, either $(\omega_e)_{e \in E_x}=(\tilde \omega_e)_{e \in E_x}$, or the configuration $(\tilde \omega_e)_{e \in E_x}$ is the complement of $(\omega_e)_{e \in E_x}$, which can be interpreted as the flip of the spin at $x$. This is possible since for any $(\alpha_e)_{e \in E_x}\in\{0,1\}^{E_x}$, 
\begin{align*}
& \Ber(p)^{\otimes E_x}((\alpha_e)_{e \in E_x})+\Ber(p)^{\otimes E_x}((1-\alpha_e)_{e \in E_x})\\
= & \Ber(1-p)^{\otimes E_x}((1-\alpha_e)_{e \in E_x})+\Ber(1-p)^{\otimes E_x}((\alpha_e)_{e \in E_x})\\
= & p^{|\alpha|}(1-p)^{4-|\alpha|}+p^{4-|\alpha|}(1-p)^{|\alpha|}.
\end{align*}

In particular, $P$ is such that there are the following possibilities for $(|\omega|,|\tilde\omega|)$:
\begin{itemize}
\item with probability $p^4+(1-p)^4$, we have $(|\omega|,|\tilde\omega|)\in\{0,4\}^2$,
\item with probability $4(p(1-p)^3+(1-p)p^3)$, we have $(|\omega|,|\tilde\omega|)\in\{1,3\}^2$,
\item with probability $6p^2(1-p)^2$, we have $|\omega|=|\tilde\omega|=2$.
\end{itemize}

\smallskip

\noindent
\begin{tabular}{|>{\centering\arraybackslash}m{4cm}||>{\centering\arraybackslash}m{1.5cm}|>{\centering\arraybackslash}m{1.5cm}||>{\centering\arraybackslash}m{2cm}|>{\centering\arraybackslash}m{1.5cm}|}

\hline

 & $(\omega_e)_{e \in E_x}$ & $(\tilde \omega_e)_{e \in E_x}$ & probability under $P$ & number of cases \\ 

\hline
\hline

$|\omega|=|\tilde\omega|=0$ \newline $(\omega_e)_{e \in E_x}=(\tilde \omega_e)_{e \in E_x}$
& \espace \begin{tikzpicture}\draw[blue, dashed] (0,0) -- (1,0) -- (1,1) -- (0,1) -- (0,0) ;\end{tikzpicture}
&  \espace \begin{tikzpicture}\draw[blue, dashed] (0,0) -- (1,0) -- (1,1) -- (0,1) -- (0,0) ;\end{tikzpicture} & 
$p^4$ 
& 1 \\
  
\hline

$|\omega|=0, |\tilde\omega|=4$ \newline $(\omega_e)_{e \in E_x}=(1-\tilde \omega_e)_{e \in E_x}$
&  \espace \begin{tikzpicture}\draw[blue, dashed] (0,0) -- (1,0) -- (1,1) -- (0,1) -- (0,0) ;\end{tikzpicture} 
&  \espace \begin{tikzpicture}\draw[blue, very thick] (0,0) -- (1,0) -- (1,1) -- (0,1) -- (0,0) ;\end{tikzpicture}  
& $(1-p)^4-p^4$  
& 1  \\

\hline

$|\omega|=|\tilde\omega|=4$ \newline $(\omega_e)_{e \in E_x}=(\tilde \omega_e)_{e \in E_x}$
&  \espace \begin{tikzpicture}\draw[blue, very thick] (0,0) -- (1,0) -- (1,1) -- (0,1) -- (0,0) ;\end{tikzpicture} 
&  \espace \begin{tikzpicture}\draw[blue, very thick] (0,0) -- (1,0) -- (1,1) -- (0,1) -- (0,0) ;\end{tikzpicture} 
& $p^4$  
& 1 \\

\hline
\hline

$|\omega|=|\tilde\omega|=1$ \newline $(\omega_e)_{e \in E_x}=(\tilde \omega_e)_{e \in E_x}$
&  \espace \begin{tikzpicture}\draw[blue, very thick] (0,0) -- (0,1) ; \draw[blue, dashed] (0,0) -- (1,0) -- (1,1) -- (0,1) ;\end{tikzpicture} 
&  \espace \begin{tikzpicture}\draw[blue, very thick] (0,0) -- (0,1) ; \draw[blue, dashed] (0,0) -- (1,0) -- (1,1) -- (0,1) ;\end{tikzpicture} 
& $p^3(1-p)$ 
& ${4 \choose 1}=4$ \\

\hline

$|\omega|=1, |\tilde\omega|=3$ \newline $(\omega_e)_{e \in E_x}=(1-\tilde \omega_e)_{e \in E_x}$
&  \espace \begin{tikzpicture}\draw[blue, very thick] (0,0) -- (0,1) ; \draw[blue, dashed] (0,0) -- (1,0) -- (1,1) -- (0,1) ;\end{tikzpicture}  
& \espace \begin{tikzpicture}\draw[blue, dashed] (0,0) -- (0,1) ; \draw[blue, very thick] (0,0) -- (1,0) -- (1,1) -- (0,1) ;\end{tikzpicture} 
& $p(1-p)^3-p^3(1-p)$ 
& ${4 \choose 1}=4$ \\

\hline

$|\omega|=|\tilde\omega|=3$ \newline $(\omega_e)_{e \in E_x}=(\tilde \omega_e)_{e \in E_x}$
& \espace \begin{tikzpicture}\draw[blue, dashed] (0,0) -- (0,1) ; \draw[blue, very thick] (0,0) -- (1,0) -- (1,1) -- (0,1) ;\end{tikzpicture} 
& \espace \begin{tikzpicture}\draw[blue, dashed] (0,0) -- (0,1) ; \draw[blue, very thick] (0,0) -- (1,0) -- (1,1) -- (0,1) ;\end{tikzpicture} 
& $p^3(1-p)$ & ${4 \choose 1}=4$ \\

\hline
\hline

$|\omega|=|\tilde\omega|=2$ \newline $(\omega_e)_{e \in E_x}=(\tilde \omega_e)_{e \in E_x}$
&  \espace \begin{tikzpicture}\draw[blue, dashed] (1,0) -- (0,0) -- (0,1) ; \draw[blue, very thick] (1,0) -- (1,1) -- (0,1) ;\end{tikzpicture} 
& \espace \begin{tikzpicture}\draw[blue, dashed] (1,0) -- (0,0) -- (0,1) ; \draw[blue, very thick] (1,0) -- (1,1) -- (0,1) ;\end{tikzpicture} 
& $p^2(1-p)^2$ 
& ${4 \choose 2}=6$\\

\hline

\end{tabular}

\medskip
Because $p<1/2$ and thus $p<1-p$, the probability measure $P$ is well defined. 
One can easily check that $P$ has the following properties.
\begin{enumerate}
\item[(P1)] The law of $(\omega_e)_{e \in E_x}$ under $P$ is $\Ber(p)^{\otimes E_x}$, and the law of $(\tilde \omega_e)_{e \in E_x}$ under $P$ is $\Ber(1-p)^{\otimes E_x}$.
\item[(P2)] $(\tilde \omega_e)_{e \in E_x}$ is more connected than $(\omega_e)_{e \in E_x}$: $P$-almost surely, if two corners of the squares are connected in $(\omega_e)_{e \in E_x}$, then they are connected in $(\tilde \omega_e)_{e \in E_x}$. 
\item[(P3)] $P$-almost surely, the parity of the degree of each corner of the square is the same in the two configuration $(\omega_e)_{e \in E_x}$ and $(\tilde \omega_e)_{e \in E_x}$.
\end{enumerate}
Note however that the coupling is not increasing: with probability $p(1-p)^3-p^3(1-p)>0$, $(\omega_e)_{e \in E_x}$ and $(\tilde \omega_e)_{e \in E_x}$ are not comparable.

\medskip
2. We now extend the previous coupling to finite boxes of $\E_2$. Define, for $n \ge 1$, $\Lambda'_n=\{x \in G_2: \;  \|x\|_\infty \le n\}$ and denote by $E(\Lambda'_n)$ the subset of edges $e\in \E_2$ such that $e_*$ has at least one end in $\Lambda_n$. Then $E(\Lambda'_n)$ is the disjoint union of the $E_x$ for $x \in \Lambda'_n$.
Set $\Delta_0= (\delta_0\otimes\delta_0)^{\otimes E_x}$.

Thus, $Q_n=P^{\otimes  \Lambda'_n} \otimes \Delta_0^{\otimes  G_2 \backslash \Lambda'_n}$ is a probability measure on $(\{0,1\}\times \{0,1\})^{\E_2}$, where  $(\omega_e,\tilde \omega_e)_{e \in \E_2} \in (\{0,1\}\times \{0,1\})^{\E_2}$ encodes two edges configurations on the whole plane: $\omega= (\omega_e)_{e \in \E_2}$ and $\tilde \omega= (\tilde \omega_e)_{e \in \E_2}$. 
From Properties (P1), (P2) and (P3), one gets:
\begin{enumerate}
\item[(P1')] The law of $\omega$ under $Q_n$ is $\Ber(p)^{\otimes E(\Lambda'_n)} \otimes \delta_0^{\otimes \E_2 \backslash E(\Lambda'_n)} $, and the law of $\tilde \omega$ under $Q_n$ is $\Ber(1-p)^{\otimes E(\Lambda'_n)}\otimes \delta_0^{\otimes \E_2 \backslash E(\Lambda'_n)}$.
\item[(P2')] $Q_n$ almost surely, if $x \stackrel{\omega}{\leftrightarrow} y$ then $x \stackrel{\tilde \omega}{\leftrightarrow} y$.
\item[(P3')] $Q_n$ almost surely, $\omega \in \OmegaEP\iff \tilde \omega \in\OmegaEP$. 
\end{enumerate}

3. Now we want to condition $Q_n$ by the event that both configurations $\omega$ and $\tilde \omega$ are even. By Property (P3'), we have
$$\overline{Q}_n(.) \stackrel{def}{=}Q_n(.|\omega \in \OmegaEP, \; \tilde \omega \in \OmegaEP)=
Q_n(.|\omega \in \OmegaEP)=Q_n(.|\tilde \omega \in \OmegaEP).$$
Remember the definition of $\mu^p_{E(\Lambda'_n),0}$. With Property (P1'), one gets
\begin{align*}
\overline{Q}_n(\omega_{E(\Lambda'_n)} \in A)
& = \frac{Q_n(\omega_{E(\Lambda'_n)} \in A,\; \omega\in\OmegaEP)}{Q_n(\omega\in\OmegaEP)} = \mu^p_{E(\Lambda'_n),0}(A).
\end{align*}
In the same manner, $\overline{Q}_n(\tilde \omega_{E(\Lambda'_n)} \in A)=\mu^{1-p}_{E(\Lambda'_n),0}(A)$. And we obtain
\begin{enumerate}
\item[(P1'')] The law of $\omega$ under $\overline{Q}_n$ is $\mu^p_{E(\Lambda'_n),0}$, and the law of $\tilde \omega$ under $\overline{Q}_n$ is $\mu^{1-p}_{E(\Lambda'_n),0}$.
\item[(P2'')] $\overline{Q}_n$ almost surely, if $x \stackrel{\omega}{\leftrightarrow} y$ then $x \stackrel{\tilde \omega}{\leftrightarrow} y$.
\end{enumerate}

4. It remains to take limits when $n$ goes to $+\infty$. 
We can extract a subsequence $(n_k)$ such that $\overline{Q}_{n_k}$ converges to a probability measure $\overline Q$ when $k$ tends to infinity. Thus
both marginals $\mu^p_{E(\Lambda'_{n_k}),0}$ and $\mu^{1-p}_{E(\Lambda'_{n_k}),0}$ also  converge when $k$ tends to infinity to the marginals of $\overline Q$. As in the proof of Theorem~1.1, their limits are Gibbs measures for even percolation, so by uniqueness, they respectively converge to $\mu_p$ and $\mu_{1-p}$. Thus,
\begin{enumerate}
\item[(P1''')] The law of $\omega$ under $\overline{Q}$ is $\mu_p$, and the law of $\tilde \omega$ under $\overline{Q}$ is $\mu_{1-p}$.
\item[(P2''')] $\overline{Q}$ almost surely, if $x \stackrel{\omega}{\leftrightarrow} y$ then $x \stackrel{\tilde \omega}{\leftrightarrow} y$.
\end{enumerate}
So the law of the field  $(\1_{\{x\communique y\}})_{(x,y)\in \Zdeux\times\Zdeux}$ under $\mu_p$ is stochastically dominated by the law of the field  $(\1_{\{x\communique y\}})_{(x,y)\in \Zdeux\times\Zdeux}$  under $\mu_{1-p}$.
\end{proof}

\subsubsection{Percolation for $p>1/\sqrt{2}$}

In the following, we give a full proof of the fact that percolation occurs for $p\ge 3/4$, and give some hints about the way to prove that there is percolation for $p> 1/\sqrt{2}$.

The proof is based on a coupling between the Ising model and the random cluster (or FK-percolation) model. We just recall a few results on the random cluster model, and refer to Grimmett's book \cite{MR2243761} for a complete survey on this model.

The random cluster measure with parameters $p$ and $q$ on a finite graph $G=(V,E)$ is the probability measure on $\{0,1\}^E$ defined by:
$$\phi^G_{p,q}(\eta)={1\over Z}{\Big({p\over 1-p}\Big)}^{\sum_{i\in E}{\eta_i}}q^{k(\eta)},$$
where $k(\eta)$ is the number of connected components in the subgraph of $G$ given by $\eta$, and $Z$ is a normalizing constant.

On $\Z^2$, it is known that at least for $p\not={\sqrt{q}\over 1+\sqrt{q}}$, there exists a unique infinite 
volume random cluster measure, that we denote by $\phi_{p,q}$ (Theorem (6.17) in \cite{MR2243761}). It is a probability measure on $\{0,1\}^{\E^2}$. 
In our study of even percolation, we use two properties of the random cluster model: its link with the Ising model, and its duality property.
For $\beta>0, \;\beta\not=\beta_c$, let us set 
\begin{equation*} 
f(\beta)=1-\exp(-2\beta). 
\end{equation*}

{(Q1)} From a spin configuration $\omega\in \{+1,-1\}^{\Z^2}$ whose distribution is any Gibbs measure~$\gamma_{\beta}$ for the Ising model with parameter $\beta\geq 0$, one obtains a subgraph $\eta\in \{0,1\}^{\E^2}$ with distribution $\phi_{f({\beta}),2}$ by keeping independently each edge between identical spins with probability $f(\beta)$, and erasing all the edges between different spins. For finite graphs, this can be found in Theorem (1.13) in \cite{MR2243761}. For the $\Zd$ case,  Theorem (4.91)  in \cite{MR2243761} says that this erasing procedure allows to couple the wired boundary infinite volume random cluster measure $\phi^1_{f({\beta}),2}$ and the Ising measure $\gamma_\beta^+$ on $\Z^2$. 

For a subgraph $\eta\in\{0,1\}^{\E^2}$, we denote by $\eta^c\in\{0,1\}^{\E^2}$ the complementary subgraph of $\Z^2$, meaning that the open edges of $\eta^c$ are exactly the closed edges of~$\eta$. We denote by $\eta_*\in \{0,1\}^{\E^2_*}$ the dual graph of $\eta$: in $\eta_*$, the edge $e_*$ is open if and only if $e$ is closed. Let us point out that $(\eta^c)_*=(\eta_*)^c$: we thus simply denote this graph by $\eta^{c}_*$. We naturally extend these notations to measures.

{(Q2)} The random cluster model has the following duality property (Theorem (6.13) in~\cite{MR2243761}): if $\eta$ is distributed according to $\phi^1_{p,2}$,  then the distribution $(\phi^1_{p,2})_*$ of $\eta_*$ is equal to $\phi^0_{p^*,2}$, where:
$${p^*\over 1-p^*}=2 \; {1-p\over p} \quad \Leftrightarrow \quad p^*={2-2p\over 2-p}.$$

Let us also recall that the measure $\mu_p$ on the edges of $\Z^2$ is obtained as the contours of any Ising measure with parameter $\beta(p)$ on $\Z^2_*$, and in particular as the contours of $\gamma^+_{\beta(p)}$ (Theorem~\ref{THEO-mesure}). Using these facts, we will prove the following proposition.

\begin{lemme}  \label{LEM:tree}
For $p\leq 1/2$, we have the following stochastic ordering:
$$\mu_p\preceq \phi^0_{2p,2}, \quad \mbox{ or equivalently, } \quad (\phi^0_{2p,2})^c\preceq \mu_{1-p}.$$
\end{lemme}

\begin{proof} For $p\leq 1/2$, starting from an Ising configuration $\Z^2_*$ of distribution $\gamma^+_{\beta(p)}$, let us draw {\bf all} the edges between identical spins. By Theorem~\ref{THEO-mesure}, the configuration on the edges of $\Z^2_*$ that we obtain is distributed according to $(\mu_{p})_*$.

By property {(Q1)} above, this measure on the edges of $\Z^2_*$ stochastically dominates the distribution $\phi^1_{f({\beta}(p)),2}$:
$$\phi^1_{f({\beta}(p)),2} \preceq (\mu_{p})_*,$$
see Figure~\ref{figure:random_cluster} for an illustration.
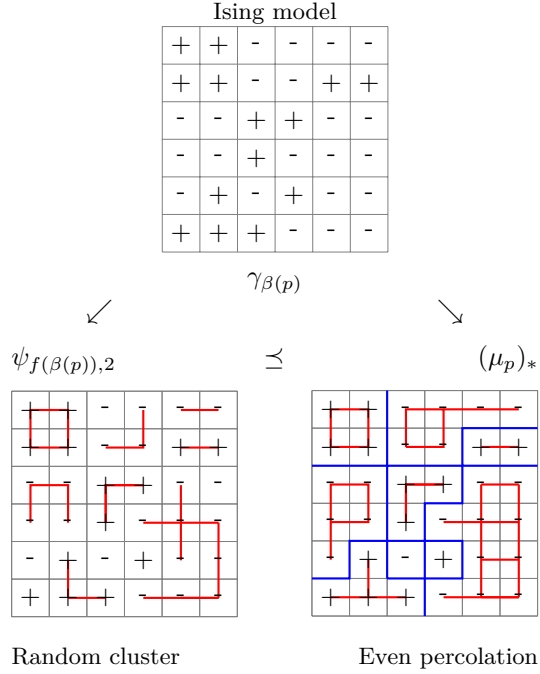
\begin{figure}
\begin{center}
{\small Ising model}

\begin{tikzpicture}[scale=0.5]
\draw[help lines] (0,0) grid (6,6);
\foreach \x in {1,2,3} \node at  (0.5,0.5+\x) {-};
\foreach \x in {2,3} \node at (1.5,0.5+\x) {-};
\foreach \x in {1,4,5} \node at (2.5,0.5+\x) {-};
\foreach \x in {0,2,4,5} \node at (3.5,0.5+\x) {-};
\foreach \x in {0,1,2,3,5} \node at (4.5,0.5+\x) {-};
\foreach \x in {0,1,2,3,5} \node at (5.5,0.5+\x) {-};
\foreach \x in {0,4,5} \node at (0.5,0.5+\x) {+};
\foreach \x in {0,1,4,5} \node at (1.5,0.5+\x) {+};
\foreach \x in {0,2,3} \node at (2.5,0.5+\x) {+};
\foreach \x in {1,3} \node at (3.5,0.5+\x) {+};
\foreach \x in {4} \node at (4.5,0.5+\x) {+};
\foreach \x in {4} \node at (5.5,0.5+\x) {+};
\end{tikzpicture}

$\gamma_{\beta(p)}$

\begin{tabular}{lcr}
\vspace{0.2cm}
\hspace{1cm}$\swarrow$& &$\searrow\hspace{1cm}$\\ \vspace{0.2cm}
$\psi_{f({\beta(p)}),2}$& $\preceq$ &$(\mu_{p})_*$\\ \vspace{0.2cm}
\begin{tikzpicture}[scale=0.5]
\draw[help lines] (0,0) grid (6,6);
\draw[help lines] (0,0) grid (6,6);
\draw[red,thick] (1.5,0.5) -- (2.5,0.5) ;
\draw[red,thick]  (3.5,0.5) -- (4.5,0.5) -- (5.5,0.5) ;
\draw[red,thick] (1.5,0.5) -- (1.5,1.5) ;
\draw[red,thick] (4.5,0.5) -- (5.5,0.5) ;
\draw[red,thick] (3.5,2.5) -- (4.5,2.5) -- (5.5,2.5) ;
\draw[red,thick] (0.5,2.5) -- (0.5,3.5) -- (1.5,3.5) -- (1.5,2.5) ;
\draw[red,thick] (0.5,4.5) -- (0.5,5.5) -- (1.5,5.5) -- (1.5,4.5) -- (0.5,4.5);
\draw[red,thick] (3.5,5.5) -- (3.5,4.5) -- (2.5,4.5);
\draw[red,thick] (2.5,2.5) -- (2.5,3.5) -- (3.5,3.5) ;
\draw[red,thick] (2.5,2.5) -- (2.5,3.5) -- (3.5,3.5) ;
\draw[red,thick] (4.5,4.5) -- (5.5,4.5) ;
\draw[red,thick] (4.5,5.5) -- (5.5,5.5) ;
\draw[red,thick] (4.5,1.5) -- (4.5,2.5) -- (4.5,3.5) ;
\draw[red,thick] (5.5,0.5) -- (5.5,1.5) -- (5.5,2.5) ;
\foreach \x in {1,2,3} \node at  (0.5,0.5+\x) {-};
\foreach \x in {2,3} \node at (1.5,0.5+\x) {-};
\foreach \x in {1,4,5} \node at (2.5,0.5+\x) {-};
\foreach \x in {0,2,4,5} \node at (3.5,0.5+\x) {-};
\foreach \x in {0,1,2,3,5} \node at (4.5,0.5+\x) {-};
\foreach \x in {0,1,2,3,5} \node at (5.5,0.5+\x) {-};
\foreach \x in {0,4,5} \node at (0.5,0.5+\x) {+};
\foreach \x in {0,1,4,5} \node at (1.5,0.5+\x) {+};
\foreach \x in {0,2,3} \node at (2.5,0.5+\x) {+};
\foreach \x in {1,3} \node at (3.5,0.5+\x) {+};
\foreach \x in {4} \node at (4.5,0.5+\x) {+};
\foreach \x in {4} \node at (5.5,0.5+\x) {+};
\end{tikzpicture}
&
&
\begin{tikzpicture}[scale=0.5]
\draw[help lines] (0,0) grid (6,6);
\draw[help lines] (0,0) grid (6,6);
\draw[blue,thick] (3,0) -- (3,2) -- (1,2) -- (1,1) -- (0,1) ;
\draw[blue,thick] (2,1) -- (4,1) -- (4,2) -- (3,2) -- (3,3) -- (4,3) -- (4,4) -- (2,4) -- (2,1) ;
\draw[blue,thick] (6,4)--(4,4)--(4,5)--(5,5)--(6,5) ;
\draw[blue,thick] (0,4)--(2,4)--(2,6) ;
\draw[red,thick] (0.5,0.5) -- (1.5,0.5) -- (2.5,0.5) ;
\draw[red,thick]  (3.5,0.5) -- (4.5,0.5) -- (5.5,0.5) ;
\draw[red,thick] (1.5,0.5) -- (1.5,1.5) ;
\draw[red,thick] (4.5,0.5) -- (5.5,0.5) ;
\draw[red,thick] (4.5,1.5) -- (5.5,1.5) ;
\draw[red,thick] (3.5,2.5) -- (4.5,2.5) -- (5.5,2.5) ;
\draw[red,thick] (0.5,1.5) -- (0.5,2.5) -- (0.5,3.5) -- (1.5,3.5) -- (1.5,2.5) -- (0.5,2.5);
\draw[red,thick] (0.5,4.5) -- (0.5,5.5) -- (1.5,5.5) -- (1.5,4.5) -- (0.5,4.5);
\draw[red,thick] (2.5,4.5) -- (2.5,5.5) -- (3.5,5.5) -- (3.5,4.5) -- (2.5,4.5);
\draw[red,thick] (2.5,2.5) -- (2.5,3.5) -- (3.5,3.5) ;
\draw[red,thick] (2.5,2.5) -- (2.5,3.5) -- (3.5,3.5) ;
\draw[red,thick] (4.5,3.5) -- (5.5,3.5) ;
\draw[red,thick] (4.5,4.5) -- (5.5,4.5) ;
\draw[red,thick] (3.5,5.5) -- (4.5,5.5) -- (5.5,5.5) ;
\draw[red,thick] (4.5,0.5) -- (4.5,1.5) -- (4.5,2.5) -- (4.5,3.5) ;
\draw[red,thick] (5.5,0.5) -- (5.5,1.5) -- (5.5,2.5) -- (5.5,3.5) ;
\foreach \x in {1,2,3} \node at  (0.5,0.5+\x) {-};
\foreach \x in {2,3} \node at (1.5,0.5+\x) {-};
\foreach \x in {1,4,5} \node at (2.5,0.5+\x) {-};
\foreach \x in {0,2,4,5} \node at (3.5,0.5+\x) {-};
\foreach \x in {0,1,2,3,5} \node at (4.5,0.5+\x) {-};
\foreach \x in {0,1,2,3,5} \node at (5.5,0.5+\x) {-};
\foreach \x in {0,4,5} \node at (0.5,0.5+\x) {+};
\foreach \x in {0,1,4,5} \node at (1.5,0.5+\x) {+};
\foreach \x in {0,2,3} \node at (2.5,0.5+\x) {+};
\foreach \x in {1,3} \node at (3.5,0.5+\x) {+};
\foreach \x in {4} \node at (4.5,0.5+\x) {+};
\foreach \x in {4} \node at (5.5,0.5+\x) {+};
\end{tikzpicture}\\
{\small Random cluster}& &{\small Even percolation}\\ \vspace{0.2cm}
\end{tabular}
\end{center}
\caption{From a configuration distributed according to $\gamma_{\beta(p)}$, we construct a configuration distributed according to $\psi_{f({\beta(p)}),2}$ by keeping each edge between identical spins with probability $f(\beta)$ (red graph on the left), and a configuration distributed according to $(\mu_{p})_*$ by keeping all edges between identical spins (red graph on the right: it is the dual graph of the blue contour graph, whose distribution is $\mu_p$).}\label{figure:random_cluster}
\end{figure}
Taking the dual of graphs, we obtain:
$$\mu_{p} \preceq (\phi^1_{f({\beta(p)}),2})_*=\phi^0_{q,2},$$
with, by property {(Q2)}, 
$(\psi_{f({\beta(p)}),2})_*=\phi_{q,2},$
with $$q=f(\beta(p))*={2-2f(\beta(p))\over 2-f(\beta(p))}={2\exp(-2\beta(p))\over 1+\exp(-2\beta(p))}={2\;{p\over 1-p}\over 1+{p\over 1-p}}=2p.$$
Thus, $\mu_{p} \preceq \phi^0_{2p,2}$. Taking the complementary of configurations, 
we obtain the second stochastic comparison.
\end{proof}

In particular, if for some parameter $p\leq 1/2$, there is percolation of closed edges in the random cluster model $\phi_{2p,2}$, then there is percolation of open edges for $\mu_{1-p}$.

Let us set 
$$\mathcal D = \{ \eta\in  \{0,1\}^{\E^2}: \; \hbox{there is an infinite cluster in } \eta^c\}.$$

The event $\mathcal D$ is non-increasing, and $p\mapsto \phi_{p,2}$ is stochastically increasing (Theorem (3.21) in~\cite{MR2243761}), so the map $p \mapsto \phi_{p,2}(\mathcal D)$ is non-increasing: there exists a critical value $\overline{p}_c(2)\in [0,1]$ such that $\phi_{p,2}(\mathcal D)>0$ for $p<\overline{p}_c(2)$ and  $\phi_{p,2}(\mathcal D)=0$ for $p>\overline{p}_c(2)$. In words, $\overline{p}_c(2)$ is the critical parameter for percolation of closed edges in the random cluster model. 
As a consequence of Lemma~\ref{LEM:tree}, we obtain the following result.

\begin{lemme}  \label{LEM:quatre}
For $p\in(1-\overline{p}_c(2)/2,1)$, $\mu_p(\C)=1$.
\end{lemme}

\begin{proof} 
By definition of $\overline{p}_c(2)$ and with Lemma \ref{LEM:tree},  $\mu_p(\C)>0$ if $2(1-p)<\overline{p}_c(2)$, which is equivalent to
$p>1-\overline{p}_c(2)/2$. We conclude with the 0--1 law.
\end{proof}

As it was first derived by Onsager \cite{MR0010315}, the critical parameter $p_{c}(2)$ for percolation of open edges in the random cluster model is equal to the self-dual point, \ie  the only fixed point of the map $p \mapsto p^*$: thus
$$p_{c}(2)=\frac{\sqrt 2}{1+\sqrt{2}}.$$

\begin{proof}[Proof of Lemma \ref{LEM:critique}]
As $p_{c,\mathrm{even}}<1/2$, Lemma \ref{LEM:tree} ensures that $\mu_{p_{c,\mathrm{even}}}\preceq \phi^0_{2p_{c,\mathrm{even}},2}$, with $2p_{c,\mathrm{even}}=p_{c}(2)$. But there is no percolation at the critical point for the free boundary condition random cluster measure in dimension 2 (Theorem (6.17) in~\cite{MR2243761}).
\end{proof}

We now give a complete and easy proof of the fact that $\overline{p}_c(2) \ge 1/2$ and a sketch of proof that the strategy used by Beffara and Duminil-Copin~\cite{MR2948685} to prove that the self-dual point of FK-percolation coincides with the critical point may be adopted here.

\begin{lemme} $\displaystyle \overline{p}_c(2) \ge 1/2$.
\end{lemme}

\begin{proof} The probability measure $\phi_{p,2}$ is dominated by a product of Bernoulli measures with parameter $p$ (Theorem (3.21) in~\cite{MR2243761}) and for $p<1/2$, the event $\mathcal D$ has a positive probability under the product of Bernoulli measures with parameter $p$. As $\mathcal D$ is a non-increasing event, the lemma follows.
\end{proof}

\begin{lemme}
$\displaystyle  \overline{p}_{c}(2)=\frac{\sqrt 2}{1+\sqrt{2}}.$
\end{lemme}
\begin{proof}[Sketch of proof]
The proof by Beffara and Duminil-Copin~\cite{MR2948685} is based of three kind of arguments:
  \begin{itemize}
  \item The study of the variation of $\phi_{p,q}(A)$ with respect to $p$.
  \item The self-duality property
  \item The FK measures are strongly associated
  \end{itemize}
  It is obvious that the self-duality property works as well for closed bonds as for open bonds. Also, the strong association of the closed bonds of FK-percolation immediately follows from the strong association of the open bonds.
  The study of the  variation of the probabilities with respect to $p$  uses the methods that are described in the monography by Grimmett~\cite{MR2243761}. It does not depend on the reference measure: they can be applied as well with $\mu(\omega)=q^{N(\omega)}$ (FK percolation) as with $\mu(\omega)=q^{N(\overline{\omega})}$ (percolation of the closed bonds of FK percolation).
\end{proof}  

From Lemma \ref{LEM:quatre}, the inequality $\overline{p}_c(2) \ge 1/2$ then implies that $\mu_p(\C)=1$ for $p>\frac34$, while $\displaystyle  \overline{p}_{c}(2)=\frac{\sqrt 2}{1+\sqrt{2}}$ implies that $\mu_p(\C)=1$ for $p>1/\sqrt{2}$.

\section{Association and monotonicity under the Eulerian condition}
\label{leschosesetranges}
    The study of Bernoulli bond percolation on a graph $G=(V,E)$ intensively uses
    the following properties of the product measure $\Ber(p)^{\otimes E}$:
    \begin{itemize}
    \item monotonicity: for every increasing event $A$, the map $p\mapsto \Ber(p)^{\otimes E}(A)$ is non-decreasing.
    \item association: for every pair of increasing event $A$, $B$, $$\Ber(p)^{\otimes E}(A\cap B)\ge \Ber(p)^{\otimes E}(A)\Ber(p)^{\otimes E}(B),$$
    or, equivalently, for every pair of non-decreasing bounded functions $F$, $G$, we have $\Cov_{\Ber(p)^{\otimes E}}(F,G)\ge 0$.
    \end{itemize}
It is natural to ask if these properties could be preserved for the measure
$$\mu_{p,G}(.)=\Ber(p)^{\otimes E}(\cdot| \text{the subgraph of open edges is Eulerian}).$$
In the following, we investigate the case of the particular undirected finite Eulerian graph~$G$ given by Figure~\ref{petitgraphhadoc}: we show that the monotonicity property is preserved whereas the association property is lost.
Note that every vertex in $G$ has even degree.    
\begin{figure}
\begin{center}
\includegraphics{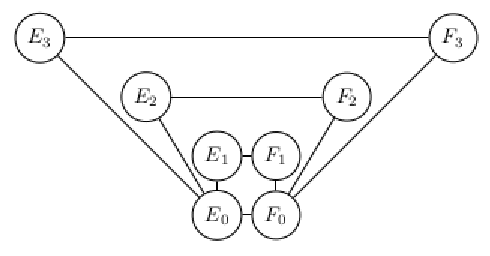}
\end{center}
  \caption{The finite Eulerian graph $G$.}\label{petitgraphhadoc}
\end{figure}        
 For simplicity, we denote $\mu_p=\mu_{p,G}$ the Eulerian percolation measure on~$G$ with opening parameter $p$.
 
For each $i\in\{0,\dots,3\}$, set $X_i=\1_{\{(E_iF_i)\text{ is open}\}}$ and $q=p/(1-p)$. It is not difficult to see that $\mu_p$ is entirely determined by
\begin{align}
\label{formuleepsilon}
\forall \epsilon=(\epsilon_i)_{1 \le i \le 3}\in \{0,1\}^3 \quad \mu_p(X_i=\epsilon_i, \; 1 \le i \le 3)=\frac{q^{N(\epsilon)}}{Z(q)},
\end{align}
where $N(\epsilon)=3(\epsilon_1+\epsilon_2+\epsilon_3)+\1_{\{\epsilon_1+\epsilon_2+\epsilon_3\text{ odd}\}}$ and $Z(q)$ is the normalizing constant:
$$Z_p=1+3q^4+3q^6+q^{10}.$$
For $i\in\{1,2,3\}$, set $C_i=\{X_0=X_i=1\}$.  These events have the same probability
$$\mu_p(C_1)=\mu_p((X_1,X_2,X_3)=(1,0,0))+\mu_p((X_1,X_2,X_3)=(1,1,1))=\frac{q^4+q^{10}}{Z_p}.$$
The events $\{X_0=X_1=X_2=X_3=1\}$,  $\{X_0=X_1=X_2=1\}$,  $\{X_0=X_2=X_3=1\}$, $\{X_0=X_1=X_3=1\}$ coincide $\mu_p$-almost-surely, so $C_1\cap C_2$ and $C_3$  are positively correlated. But
\begin{align*}
 \mu_p(C_1\cap C_2)-\mu_p(C_1)\mu_p(C_2)&=\frac{q^{10}}{Z_p}-\left(\frac{q^4+q^{10}}{Z_p}\right)^2
  =\frac{-q^8+q^{10}+q^{14}+3q^{16}}{Z_p^2}<0
\end{align*} 
for $q<0,74$, so $C_1$ and $C_2$ are negatively correlated for each $p<0,42$.

However, the sequence $(\mu_p)_{p\in [0,1]}$ is non-decreasing for the stochastic order:
\begin{theorem}
Let $G=(V,E)$ be the graph illustrated by Figure \ref{petitgraphhadoc}.
Let $A\in\mathcal{P}(\{0,1\}^{E})$ be an increasing event.
Then, $p\mapsto \mu_p(A)$ is non-decreasing.
Equivalently, if $F$ is a monotonic boolean function on $\{0,1\}^{|E|}$, $p\mapsto \int F\ d\mu_p$ is non-decreasing.
\end{theorem}

\begin{proof}
Let $\eta_{i,j}\in\{0,1\}$ be the state of the edge between vertices $i$ and $j$. 
The structure of the graph implies that $\mu_p$-almost surely, 
$$
\forall i \in \{1,2,3\} \quad X_i =\eta_{E_0,E_i}=\eta_{E_i,F_i}=\eta_{F_i,F_0}.
$$
  So, if $F$ is a non-decreasing fonction on $\{0,1\}^E$, we have $\mu_p$ a.s. :
\begin{align*}
&F(\eta_{E_0,F_0},\eta_{E_0,E_1},\eta_{E_1,F_1},\eta_{F_1,F_0},\eta_{E_0,E_2},\eta_{E_2,F_2},\eta_{F_2,F_0},\eta_{E_0,E_3},\eta_{E_3,F_3},\eta_{E_3,F_0})\\=&F_1(X_0,X_1,X_2,X_3),\text{ with }F_1(x,y,z,t)=F(x,y,y,y,z,z,z,t,t,t).
\end{align*}
By construction, $F_1$ is a non-decreasing function, so it is sufficient to prove
that for any non-decreasing function $F$:$\{0,1\}^4\to\{0,1\}$, the map $p\mapsto \int F((X_i)_{0 \le i \le 3})\ d\mu_p$ is non-decreasing.
The law of $(X_0,X_1,X_2,X_3)$ under $\mu_p$ is easy to express: for every $\epsilon=(\epsilon_i)_{0 \le i \le 3}\in \{0,1\}^4,$
\begin{align*}
&\mu_p(X_i=\epsilon_i, \; 0 \le i \le 3) = \1_{\{\epsilon_0+\epsilon_1+\epsilon_2+\epsilon_3\text{ even}\}} \mu_p(X_i=\epsilon_i, \; 1 \le i \le 3).
\end{align*}
  With \eqref{formuleepsilon}, it is easy to see that $\int F((X_i)_{0 \le i \le 3})\ d\mu_p$
  can be expressed as a rational function of $q=\frac{p}{1-p}$:
$$\int F((X_i)_{0 \le i \le 3})\ d\mu_p=\frac{P_F(q)}{Z(q)},$$
  so if is sufficient to check that the polynomial $R_F=P_F'Z-P_FZ'$ has no positive root, which can be easily performed with a modern computer.
  In fact, it happens that for each monotonic boolean functions $F$, 
  $$R_F\in\left\{\begin{array}{ll}
  0,10q^{9}+ 18q^{5}+ 12q^{3},&12q^{13}+ 22q^{9}+ 18q^{5}+ 4q^{3},\\
  12q^{15}+ 12q^{13}+ 4q^{9}+ 4q^{3},&12q^{15}+ 18q^{13}+ 10q^{9},\\
  12q^{15}+ 6q^{13}- 2q^{9}+ 8q^{3},&12q^{15}- 8q^{9}+ 12q^{3},\\
  18q^{13}+ 28q^{9}+ 18q^{5},&4q^{15}+ 12q^{13}+ 16q^{9}+ 12q^{5}+ 4q^{3},\\
  4q^{15}+ 18q^{13}+ 22q^{9}+ 12q^{5}&4q^{15}+ 4q^{9}+ 12q^{5}+ 12q^{3},\\
  4q^{15}+ 6q^{13}+ 10q^{9}+ 12q^{5}+ 8q^{3},&6q^{13}+ 16q^{9}+ 18q^{5}+ 8q^{3},\\
  8q^{15}+ 12q^{13}+ 10q^{9}+ 6q^{5}+ 4q^{3},&8q^{15}+ 18q^{13}+ 16q^{9}+ 6q^{5},\\
  8q^{15}- 2q^{9}+ 6q^{5}+ 12q^{3},&8q^{15} + 6q^{13} + 4q^9 + 6q^5 + 8q^3
  \end{array}\right\}$$ In most cases, the coefficients of $R_F$ are non-negative; in any case, it is easy to prove that $R_F$ has no positive root.

  We obtain the list of the 168 functions by a  brute-force algorithm based on the following remark: if $M_n$ denotes the set of monotonic boolean functions on $\{0,1\}^n$, there is a natural one-to-one correspondance between
  $M_{n+1}$ and $\{(f,g)\in M_n^2; f\le g\}$: a function $G$ of $n+1$ variables $(x_1,\dots,x_{n+1})$ is associated to the pair of functions
  $((x_1,\dots,x_n)\mapsto F(x_1,\dots,x_n,0),(x_1,\dots,x_n)\mapsto F(x_1,\dots,x_n,1))$.
The number $|M_n|$ of monotonic boolean functions is known as the Dedekind number. The sequence $(|M_n|)_{n\ge 1}$ increases very fast and is not easy to compute. In fact, the exact values are only known for $n\le 8$ (see Wiedemann~\cite{MR1129608}). 
\end{proof}
        
We conjecture that this result should be more general:
\begin{conj}
Let $G=(V,E)$ be a Eulerian graph.
  Then, the sequence of Eulerian percolation measures $(\mu_p)_{p\in [0,1]}$ on $\{0,1\}^E$ is stochastically non-decreasing.
\end{conj}
Note that Cammarota and Russo~\cite{MR1115955} proved related results supporting this conjecture.
\section*{Appendix: code of the Julia program}
\begin{small}
\begin{lstlisting}
using SymPy

function valeur(a,b,c,d)
    q=Sym("q")
    if (((a+b+c+d)%2)==0)
        n=3*(b+c+d)+a
        if (n>0)
            return poly(q^n)
        else return 1 end
    else return 0
    end
end

function zp()
    q=Sym("q")
    z=poly(q)*0
    for a=0:1
        for b=0:1
            for c=0:1
                for d=0:1
                    z+=valeur(a,b,c,d)
                end
            end
        end
    end
    return z
end

function evalue(numero,t)
taille=length(t)
if (taille==0) return(Int32(numero))
else
 haute=Int32(numero) & (2^(2^(taille-1))-1)
 basse=div((Int32(numero)-haute),2^(2^(taille-1)))
 if (t[taille]==1)
 return evalue(haute,t[(1:taille-1)'])
 else
 return evalue(basse,t[(1:taille-1)'])
end
end
end

function integrale_num(numero)
    p=Sym("q")
    z=poly(p)*0
    for a=0:1
        for b=0:1
            for c=0:1
                for d=0:1
                    z+=evalue(numero,[a b c d])*valeur(a,b,c,d)
                end
            end
        end
    end
    return z
end   

function variation(numero)
    Z=zp()
    dZ=diff(Z)
    N=integrale_num(numero)
    dN=diff(N)
    nder=dN*Z-dZ*N
    return nder
end

function compte(numero)
    d=variation(numero)
    p=Sym("q")
return count_roots(d,0)-1*(subs(d,p,0)==0)
end

t=[Int64(0);Int64(1)]
nbit=Int32(1)
for i=1:4
s=[]
for a=t
for b=t
if ((a & b)==a)
  s=[s;  a*(2^nbit)+b]
   end
end
end
nbit=nbit*2
print("For n=",i,", there is ",length(s),"functions. Here is the list:")
println(s); t=s
end


bad=0
for i=t
    j=compte(i)
   if (j>0)
    bad+=1
    end
    print(" ",integrale_num(i)," ",compte(i),"\n")

end
print("Z(q)=",zp(),"\n")
print("Among the ", length(t), "non-decreasing fonctions ,",bad)
print(" do not satisfy the conjecture.")             
\end{lstlisting}
\end{small}

\def\refname{References}
\bibliographystyle{plain}
\bibliography{percopaire}

\end{document}